\documentclass{amsart}

\usepackage{mathabx}

\usepackage{mathrsfs}
\usepackage{amsfonts}

\usepackage{caption}

\usepackage{cases}

\usepackage[hyperindex]{hyperref}
\usepackage{cite}
\usepackage{graphicx}
\usepackage{color}

\usepackage{mathtools}

\newcommand{\balg}{\begin{algorithm}}
\newcommand{\ealg}{\end{algorithm}}
\newcommand{\br}{\begin{remark}}
\newcommand{\er}{\end{remark}}
\newcommand{\bex}{\begin{example}}
\newcommand{\eex}{\end{example}}

\newtheorem{theorem}{Theorem}[section]
\newtheorem{lemma}[theorem]{Lemma}
\newtheorem{proposition}[theorem]{Proposition}

\theoremstyle{definition}

\newtheorem{example}[theorem]{Example}

\newtheorem{algorithm}[theorem]{Algorithm}

\newtheorem{remark}[theorem]{Remark}

\numberwithin{equation}{section}

\begin{document}

\title{\textcolor[rgb]{1.00,0.00,0.00}{Interiors of completely positive cones}}


\author{Anwa Zhou}
\address{Department of Mathematics, Shanghai Jiao Tong University,
Shanghai 200240, P.R. China}
\email{congcongyan@sjtu.edu.cn}
\thanks{}

\author{Jinyan Fan$^*$}
\address{Department of Mathematics, and MOE-LSC, Shanghai Jiao Tong University,
Shanghai 200240, P.R. China}
\email{jyfan@sjtu.edu.cn}
\thanks{$^*$ The corresponding author. The work is partially supported by NSFC 11171217.}


\date{}

\dedicatory{}

\begin{abstract}
A symmetric matrix $A$ is completely positive (CP) if there exists an entrywise
nonnegative matrix $B$ such that $A = BB^T$.
We characterize the interior of the CP cone.
A semidefinite algorithm is proposed for checking interiors of the CP cone, and its properties are studied.
A CP-decomposition of a matrix in Dickinson's form can  be obtained if it is an interior of the CP cone.
Some computational experiments are also presented.

\end{abstract}


\maketitle

\section{Introduction}
A real $n\times n$ symmetric matrix $A$ is   completely
positive (CP) if there exist nonnegative vectors $b_1,\cdots,b_m
\in \mathbb{R}^n$ such that
\begin{equation}\label{CPd}
A = b_1 b_1^T +\cdots +b_m b_m^T,
\end{equation}
where $m$ is called the  length of the decomposition
(\ref{CPd}). The smallest $m$ in the above is called the
 CP-rank of $A$. If $A$ is CP, we call (\ref{CPd}) a
 CP-decomposition of $A$. Clearly, $A$ is CP if and only if
$A = B B^T$ for an entrywise nonnegative matrix $B$.
Hence, a CP-matrix is not only positive semidefinite but also nonnegative entrywise.


Let $\mathcal{S}_n$ be the space of real  $n\times n$ symmetric matrices.
For a  cone $\mathcal{C}\subseteq \mathcal{S}_n$, its dual cone is defined as
$$\mathcal{C}^*:= \{G \in \mathcal{S}_n : \langle A,G\rangle \geq 0 \; \text{for all} \; A \in \mathcal{C}\},$$
where $\langle A,G\rangle=\text{trace}(AG)$ is the trace inner product.
Denote
\begin{align*}
& \mathcal{C}_n = \{A \in \mathcal{S}_n : A= BB^T \ \text{with}\ B\geq0 \}, \text{the completely positive cone},\\
& \mathcal{C}_n^* = \{ G \in \mathcal{S}_n: x^TGx\geq 0\ \text{for all}\ x\geq 0\},
\text{the copositive cone}.
\end{align*}
Both $\mathcal {C}_n$ and $\mathcal{C}_n^*$
are proper cones (i.e. closed, pointed, convex and
full-dimensional).  Moreover, they are  dual to each other \cite{Hall1967}.

 The completely positive cone and copositive cone have wide applications in mixed binary quadratic programming \cite{Burer},
standard quadratic optimization problems
and general quadratic programming \cite{Bomze}, etc.
Some NP-hard problems can also be formulated as linear optimization
problems over the CP cone and the copositive cone (cf. \cite{deKlerk}).
We refer to \cite{BermanN,Bomze1,Dur}
for the work in the field.

The membership problems for the completely positive cone and the copositive cone are NP-hard (cf. \cite{Dickinson11,AnstreicherB}). To compute a CP-decomposition of a CP-matrix is also hard. Dickinson \& D\"{u}r \cite{Dickinson2012} studied the CP-checking and
CP-decomposition of some sparse matrices.
Sponseldur \& D\"{u}r \cite {Sponseldur} used polyhedral
approximations to  project a matrix to  $\mathcal{C}_n$;
a CP-decomposition of a matrix can  be obtained if it is an interior of $\mathcal{C}_n$.
In \cite{ZhouF}, a semidefinite algorithm is proposed for solving the CP-matrix completion problem, which includes the CP-checking as a special case;
a CP-decomposition for a general CP-matrix can be found by the algorithm.

Denote int$(\mathcal{C}_n)$ and bd$(\mathcal{C}_n)$ the interior and the boundary of  $\mathcal{C}_n$, respectively.
Shaked-Monderer, Bomze, Jarre \&  Schachinger \cite{Shaked-MondererBomze} showed  that the maximal CP-rank of $n\times n$ CP-matrices is attained at a positive definite matrix on bd$(\mathcal{C}_n)$. Denote $\mathbb{R}^n_+ := \{x \in \mathbb{R}^n \mid x \geq 0\}$
and $\mathbb{R}^n_{++} := \{x \in \mathbb{R}^n \mid x > 0\}$.
D\"{u}r \& Still \cite{DurS} characterized  int$(\mathcal{C}_n)$ as:
\begin{align}
  \text{int}(\mathcal{C}_n)  = &\{BB^T \mid B=(B_1, B_2) \; \text{with}\;B_1>0\  \text{nonsingular}, \;  B_2 \geq 0\}\nonumber\\
     = & \left\{\left.\sum^{m}_{i=1} {b_i b_i^T} \right| \begin{array}{c}
                                             m\geq n, b_i \in \mathbb{R}^n_+ \;\text{for all} \; i, \\
                                             \{b_1,\ldots,b_n\}\subseteq \mathbb{R}^n_{++}, \\
                                            \text{span}\{b_1,\ldots,b_n\}=\mathbb{R}^n
                                          \end{array}\right\}.\label{Durform}
\end{align}
Dickinson \cite{Dickinson2010} further characterized  int$(\mathcal{C}_n)$ as:
\begin{align}
  \text{int}(\mathcal{C}_n)  = & \{BB^T \mid \text{rank} (B)=n, B=(b_1, \tilde B), b_1\in \mathbb{R}^{n}_{++}, \tilde B \geq 0\}\nonumber \\
     = & \left\{ \left.\sum^{m}_{i=1} {b_i b_i^T}
 \right| \begin{array}{c}
                                             b_1\in \mathbb{R}^{n}_{++}, b_i\in \mathbb{R}^n_+ \;\text{for} \; i=2,\cdots,m, \\
                                            \text{span}\{b_1,\cdots,b_m\}=\mathbb{R}^n
                                          \end{array}\right\}. \label{dickinsonform}
\end{align}
The above characterizations are very useful in checking interiors of $\mathcal{C}_n$.

How do we check whether a matrix is in the interior of $C_n$ if
it is not given in D\"{u}r \& Still's form \eqref{Durform} or  Dickinson's form \eqref{dickinsonform}?
Little is known for checking interiors or boundaries of  $\mathcal{C}_n$.
In this paper, we characterize interiors of $\mathcal{C}_n$ from the view of optimization.
A semidefinite algorithm is proposed to check
whether a symmetric matrix $A\notin \mathcal{C}_n$, or $A\in \text{bd}(\mathcal{C}_n)$, or $A\in \text{int}(\mathcal{C}_n)$.
If $A\notin \mathcal{C}_n$, we can get a certificate.
If $A\in \mathcal{C}_n$, we can get a CP-decomposition of $A$.
Moreover, a CP-decomposition in Dickinson's form can also be obtained by a similar algorithm.

The paper is organized as follows. In Section 2, we give a new sufficient and necessary condition
to characterize interiors of $\mathcal{C}_n$.
In Section 3, we formulate the problem of checking the membership and interiors of $\mathcal{C}_n$ as the linear optimization with moments.
 In Section 4, we present a semidefinite algorithm for the problem. Its basic properties are also studied.
Some computational results are reported in Section 5.
Finally in Section 6, we discuss how to give a CP-decomposition of a matrix in Dickinson's form if it is an interior of $\mathcal{C}_n$.


\section{A Characterization of interiors}

In this section, we characterize interiors of $\mathcal{C}_n$ from the view of optimization. 

\begin{lemma}\label{iffintDur} Suppose $A\in \mathcal{S}_n$. Then $A \in \text{int}(\mathcal{C}_n)$ if and only if
for some $C \in \text{int}(\mathcal{C}_n)$,
there exists a $\lambda>0$   such that $A- \lambda C \in \mathcal{C}_n$.
\end{lemma}

\begin{proof} Given $A\in \text{int}(\mathcal{C}_n)$, then there exists a $\delta>0$ such that
for any $D\in \mathcal{S}_n$ with $\|A-D\|\leq \delta$, we have $D\in \mathcal{C}_n$.
Choose an arbitrary $C \in \text{int}(\mathcal{C}_n)$.
Obviously, $C$ is positive and nonsingular.
Let $\lambda= \delta/\|C\|$.
Then $\|A-(A- \lambda C)\|\leq \delta$, which implies that $A- \lambda C\in \mathcal{C}_n$.

Conversely, suppose $C \in \text{int}(\mathcal{C}_n)$ and $ A- \lambda C \in \mathcal{C}_n$, where $\lambda>0$.
By \eqref{Durform}, there exist
$B_1>0$ nonsingular and $B_2 \geq 0$ such that $C=(B_1, B_2)(B_1, B_2)^T$.
Meanwhile, there exists a $B_3\geq 0$ such that $A- \lambda C=B_3B_3^T$. Hence,
$$
A=\lambda C+ B_3B_3^T=(\sqrt{\lambda}B_1,\sqrt{\lambda}B_2,B_3)(\sqrt{\lambda}B_1,\sqrt{\lambda}B_2, B_3)^T.
$$
So, by \eqref{Durform}, $A \in \text{int}(\mathcal{C}_n)$.
\end{proof}

Lemma \ref{iffintDur} gives an equivalent condition for a  matrix $A$
to be an interior of $\mathcal{C}_n$.
We wonder how to compute such a $\lambda$.
This can be done by solving the linear optimization problem:
\begin{equation*}
 (P_1):\qquad  \left\{ \begin{array}{lll}
  f_1^*=&\max  & \lambda\\
  &\mbox{s.t.} & A-\lambda C \in \mathcal{C}_n
 \end{array} \right.
\end{equation*}
for some given $C\in \text{int}(\mathcal{C}_n)$.
A simple choice of $C$ is $I_n+E_n$.
Here,  $I_n$ denotes the $n\times n$ identity matrix and $E_n$ the $n\times n$ matrix of all ones.
By Lemma \ref{iffintDur},
if $f_{1}^*> 0$, then  $A \in \text{int}(\mathcal{C}_n)$;
if $f_{1}^*= 0$, then  $A \in \text{bd}(\mathcal{C}_n)$;
if $f_{1}^*< 0$, then  $A \notin \mathcal{C}_n$.

Since $\mathcal{C}_n$ and $\mathcal{C}_n^*$ are dual to each other, we know
\begin{align}\label{notin1}
   A\notin \mathcal{C}_n \Longleftrightarrow  \exists X \in \mathcal{C}_n^* \; \text{such that}\; \langle A,X \rangle < 0.
\end{align}
Hence, $A\notin \mathcal{C}_n$ if and only if there exists $X \in \mathcal{C}_n^*$ satisfying
\begin{align}\label{notin2}
   \langle A,X \rangle<0, \quad \langle X, C \rangle=1.
 \end{align}
On the  other hand, as shown in \cite{Dickinson2010, Berman},
\begin{align}\label{int1}
  A \in \text{int}(\mathcal{C}_n)   \Longleftrightarrow  \langle A,X \rangle > 0\ \text{for all} \ X \in \mathcal{C}_n^* \setminus \{0\}.
  \end{align}
Hence, $A\in \mathcal{C}_n$ if and only if for all $X \in \mathcal{C}_n^*$ with $\langle X, C \rangle=1$,
     \begin{align}\label{int2}
  \langle A,X \rangle >0.
 \end{align}
Therefore, checking interiors of $\mathcal{C}_n$ is equivalent to solving
the   linear optimization problem over $\mathcal{C}_n^*$:
\begin{equation*}
(D_1):\qquad   \left\{ \begin{array}{lll}
  g_1^*=&\min  & \langle A,X \rangle\\
  &\mbox{s.t.} & \langle X, C\rangle =1\\
  &            & X \in \mathcal{C}_n^*.
 \end{array}\right.\qquad
\end{equation*}
By \eqref{notin2} and \eqref{int2}, if $g_{1}^*> 0$, then  $A \in \text{int}(\mathcal{C}_n)$;
if $g_{1}^*= 0$, then  $A \in \text{bd}(\mathcal{C}_n)$;
if $g_{1}^*< 0$, then  $A \notin \mathcal{C}_n$.

In fact, the optimization problems ($P_1$) and ($D_1$) are dual to each other.
Denote by Feas($P$) the feasible region of an optimization problem ($P$).
By the standard duality theory, we have  $g_1^* \geq f_1^*$  for all $X\in \text{Feas}(D_1)$ and $\lambda \in \text{Feas}(P_1)$.
This is referred to as weak duality.
If there exists a $\lambda \in \text{Feas}(P_1)$ such that $A-\lambda C \in \text{int}(\mathcal{C}_n)$, we say that Slater's condition holds for ($P_1$) and $\lambda$ is a strictly feasible point of ($P_1$).
If there exists a $X \in \text{Feas}(D_1) \cap \text{int}(\mathcal{C}_n^*)$,
we say that Slater's condition holds for ($D_1$) and $X$ is a strictly feasible point of ($D_1$).
Under Slater's conditions, strong duality holds (i.e. $g_1^* = f_1^*$).

The optimization problems ($P_1$) and  ($D_1$) have the following properties.

\begin{theorem}\label{strongdualP}
Suppose $A\in \mathcal{S}_n$ and $C \in \text{int}(\mathcal{C}_n)$.
Then the optimums of ($P_1$) and ($D_1$) are finite and equal,
and the optimal solution sets of ($P_1$) and ($D_1$) are nonempty.
Furthermore, if $f_1^*<0$,
then $A \notin \mathcal{C}_n$; if $f_1^*=0$,
then $A \in \text{bd}(\mathcal{C}_n)$;
 if $f_1^*>0$, then $A \in \text{int}(\mathcal{C}_n)$.
\end{theorem}

\begin{proof}
We first show that Slater's condition holds for ($P_1$).
If $A=0$, then all $\lambda<0$ are strictly feasible points of ($P_1$).
If  $A\not=0$, due to  $C \in \text{int}(\mathcal{C}_n)$,
there exists a $\delta>0$ such that $D\in \text{int}(\mathcal{C}_n)$
for all $D\in \mathcal{S}_n$ with $\|C-D\|\leq \delta$.
Let $\lambda\leq -\frac{\|A\|}{\delta}$.
As $\|C-(C-\frac{1}{\lambda}\cdot A)\|\leq \delta$,
we have  $C-\frac{1}{\lambda} A\in \text{int}(\mathcal{C}_n)$.
So,  $A-\lambda C\in \text{int}(\mathcal{C}_n)$.
That is,  $\lambda$ is a strictly feasible point of ($P_1$).

Choose an arbitrary $P \in \text{int}(\mathcal{C}_n^*)$.
Since $C \in \text{int}(\mathcal{C}_n)$,
we have $\langle P, C\rangle>0$.
Thus, $\langle P, C\rangle^{-1} P \in \text{Feas}(D_1) \cap \text{int}(\mathcal{C}_n^*)$.
So, Slater's condition holds for ($D_1$).

It is obvious that the optimum of ($P_1$) is finite.
Therefore, the optimums of ($P_1$) and ($D_1$) are finite and equal, and the
optimal solution sets of ($P_1$) and ($D_1$) are both nonempty
by the duality theory given in \cite[Theorems 1.25 and 1.26]{DickinsonPhd}.

By Lemma \ref{iffintDur}, we obtain the rest part of the theorem.
\end{proof}

Therefore, checking  interiors of $\mathcal{C}_n$ is equivalent to solving ($P_1$) or ($D_1$). For all  $A\in \mathcal{S}_n$ and $C \in \text{int}(\mathcal{C}_n)$, a maximizer $\lambda^*$ of ($P_1$)
always exists. This leads to an interesting result for $A-\lambda^* C$.

\begin{proposition}\label{bdminimizer}
Suppose $A\in \mathcal{S}_n$, $C \in \text{int}(\mathcal{C}_n)$, and $\lambda^*$ is a maximizer of ($P_1$).
Then $A-\lambda^* C \in \text{bd}(\mathcal{C}_n)$.
\end{proposition}

\begin{proof}
We prove by contradiction. Obviously, $A-\lambda^* C \in  \mathcal{C}_n$.
 Suppose $A-\lambda^* C \in \text{int}(\mathcal{C}_n)$.
 Then, there exists a  $\delta>0$ such that $D\in \mathcal{C}_n$ for all $D\in \mathcal{S}_n$ with $\|A-\lambda^* C-D\|\leq \delta$.
 Hence, $A-(\lambda^*+\varepsilon) C \in \mathcal{C}_n$ for all $0<\varepsilon\leq \delta/\|C\|$.
 Thus $\lambda^*+\varepsilon$ is a feasible point of ($P_1$), which contradicts that
$\lambda^*$ is the maximizer of ($P_1$). The proof is completed.
\end{proof}

 \section{A linear moment optimization approach}

As shown above, checking interiors of  $\mathcal{C}_n$ is equivalent to a linear optimization problem with $\mathcal{C}_n$. Generally, it is difficult to solve it directly.
In this section, we formulate ($P_1$) as a linear optimization problem with the cone of moments. We begin with some basics about moments.

\subsection{Formulation as a moment problem}\label{famp}
A symmetric matrix can be identified by a vector consisting of its upper triangular entries.
Let $\mathbb{N}$ be the set of nonnegative integers.
For $\alpha = (\alpha_1,\cdots, \alpha_n) \in \mathbb{N}^n$,
denote $|\alpha| := \alpha_1+\cdots+\alpha_n$.
Let
\begin{equation}\label{AE} \mathcal {A} := \{\alpha
\in \mathbb{N}^n:\, \alpha=e_i+e_j, j\geq i, i,j=1,\cdots,n \},
\end{equation}
where $e_i$ is the $i$-th unit  vector.
So, a matrix $A\in \mathcal{S}_n$ can be identified as a vector $a$ as:
$$
a =(a_{\alpha})_{\alpha\in \mathcal{A}} \in \mathbb{R}^\mathcal{A},\quad
a_{\alpha}=A_{ij}\ \mbox{if}\ \alpha=e_i+e_j,
$$
where $\mathbb{R}^{\mathcal {A}}$ denotes the
space of real vectors indexed by  $\alpha \in \mathcal {A}$.
We call $a$ an $\mathcal {A}$-truncated moment sequence ($\mathcal {A} $-tms).
Let
\begin{equation} \label{KE}
K=\{x\in \mathbb{R}^n :\,
x_1^2+\cdots +x_n^2 -1= 0, x_1 \geq 0, \cdots, x_n \geq 0\}
\end{equation}
be the nonnegative part of the unit sphere.
 Every nonnegative vector is a multiple
of a vector in  $K$.
So, by (\ref{CPd}),
$A\in \mathcal{C}_n$ if and only if
there exist vectors
$b_1,\cdots,b_m \in K$ and $\rho_1,\cdots,\rho_m >0$ such that
 \begin{equation}\label{ECPe}
A=\rho_1 b_1 b_1^T +\cdots +\rho_m b_m b_m^T.
 \end{equation}

 The $\mathcal {A}$-truncated $K$-moment problem
($\mathcal {A}$-T$K$MP) studies whether or not a given {$\mathcal {A}$-tms $a$ admits a
$K$-measure} $\mu$, i.e., a nonnegative Borel measure $\mu$ supported
in $K$ such that
\[
a_{\alpha}= \int_K x^{\alpha} d \mu, \quad
\forall\, \alpha \in  \mathcal {A},
\]
where $x^{\alpha} := x^{\alpha_1}_1 \cdots x^{\alpha_n}_n$.
A measure $\mu$ satisfying the above is called a
{$K$-representing measure} for $a$. A measure is called {finitely
atomic} if its support is a finite set, and is called $m$-atomic if its
support consists of at most $m$ distinct points.
We refer to \cite{Nie} for representing measures of truncated moments sequences.

Hence, by \eqref{ECPe}, a symmetric matrix $A$, with the identifying vector
$a \in \mathbb{R}^{\mathcal {A}}$,  is completely positive if and only if $a$ admits
an $m$-atomic $K$-measure,  i.e.,
\begin{equation}
\label{ECPee} a=\rho_1 [b_1]_{\mathcal {A}} +\cdots +\rho_m [b_m]_{\mathcal
{A}},
\end{equation}
where each $b_i \in K$, $\rho_i>0$, and
\[
[b]_{\mathcal {A}} :=(b^{\alpha})_{\alpha \in
{\mathcal {A}}}.
\]
In other words,
checking CP is equivalent to an $\mathcal {A}$-T$K$MP with
$\mathcal {A}$ and $K$ given in (\ref{AE}) and (\ref{KE}) respectively.

Denote
\[
\mathbb{R}[x]_{\mathcal {A}}:= \mbox{span}\{x^{\alpha}: \alpha\in \mathcal {A}\}.
\]
We say $\mathbb{R}[x]_{\mathcal {A}}$ is {$K$-full} if there exists a polynomial $p \in \mathbb{R}[x]_{\mathcal {A}}$ such
that $p|_{K} >0$ (cf. \cite{Nie4}).
An $\mathcal {A}$-tms $a \in \mathbb{R}^{\mathcal {A}}$
defines an $\mathcal {A}$-Riesz functional $\mathscr{L}_{a}$ acting on
$\mathbb{R}[x]_{\mathcal {A}}$ as
 \begin{equation}\label{La} \mathscr{L}_{a}
(\sum_{\alpha\in \mathcal {A}}p_{\alpha} x^{\alpha}):=
\sum_{\alpha\in \mathcal {A}}p_{\alpha} a_{\alpha}.
 \end{equation}
 For convenience, we also denote $\langle p,a \rangle:= \mathscr{L}_{a}(p)$.
Let
$$
\mathbb{N}_d^n := \{ \alpha \in \mathbb{N}^n: \, |\alpha| \leq d\}$$
 and
$$
\mathbb{R}[x]_{d}:=\mbox{span}\{x^{\alpha}: \alpha\in \mathbb{N}^n_{d}\}.
$$
For $s \in \mathbb{R}^{\mathbb{N}^n_{2k}}$ and $q
\in \mathbb{R}[x]_{2k}$, the $k$-th localizing matrix  of
$q$ generated by $s$ is the symmetric
matrix $L^{(k)}_q (s)$ satisfying
 \begin{equation}\label{Lzqp2}
\mathscr{L}_s (q p^2)= vec(p)^T (L^{(k)}_q (s))vec(p), \quad \forall p\in
\mathbb{R}[x]_{k-\lceil deg(q)/2 \rceil}.
 \end{equation}
In the above, $vec(p)$ denotes the coefficient vector of $p$ in the
graded lexicographical ordering, and $\lceil t\rceil$
denotes the smallest integer that is not smaller than $t$.
In particular, when $q=1$, $L^{(k)}_1 (s)$ is called a  $k$-th
order moment matrix  and denoted as $M_k(s)$. We refer to \cite{Nie,Helton,Nie4} for
more details about localizing and moment matrices.

Denote the polynomials:
\begin{align*}
h(x):=  x_1^2+\cdots +x_n^2-1,g_0(x):=1,
 g_1(x): =  x_1, \cdots,  g_n(x): = x_n.
\end{align*}
 Note that $K$ given in (\ref{KE}) is nonempty compact.
  We can also describe $K$  equivalently as
\begin{equation}\label{K2}
K=\{x\in \mathbb{R}^n: \ h(x)= 0, g(x) \geq 0\},
\end{equation}
where $g(x)=(g_0(x), g_1(x),\cdots,g_{n}(x))$.
As shown in \cite{Nie}, a necessary condition for $s \in \mathbb{R}^{\mathbb{N}^n_{2k}}$
to admit a $K$-measure is
\begin{equation}
\label{SDPC}
  L^{(k)}_{h} (s) = 0, \quad \mbox{and}\quad L^{(k)}_{g_j} (s) \succeq
0, \quad j=0,1,\cdots,n.
\end{equation}
If, in addition to (\ref{SDPC}), $s$ satisfies the {rank condition}
\begin{equation}
\label{RC}
\text{rank} M_{k-1}(s) =\text{rank} M_{k} (s),
\end{equation}
then $s$ admits a unique
$K$-measure, which is $\text{rank} M_k(s)$-atomic
(cf. Curto and Fialkow \cite{CurtoF}).
We say that $s$ is {flat} if both (\ref{SDPC}) and (\ref{RC}) are satisfied.

Given two tms' $y \in \mathbb{R}^{\mathbb{N}^n_{d}}$ and $z \in
\mathbb{R}^{\mathbb{N}^n_{e}}$, we say $z$ is an  extension  of $y$, if $d\leq e$ and
$y_{\alpha} = z_{\alpha}$ for all $\alpha \in \mathbb{N}^n_{d}$. We denote
by $z|_{\mathcal {A}}$ the subvector of $z$, whose entries are indexed by
$\alpha \in \mathcal {A}$. For convenience, we denote by $z|_{d}$
the subvector $z |_{\mathbb{N}^n_{d}}$.
If $z$ is flat and extends $y$, we say $z$ is a {flat
extension} of $y$. Note that an $\mathcal {A}$-tms
$a \in \mathbb{R}^{\mathcal{A}}$ admits a $K$-measure
if and only if it is extendable to a flat tms $z \in \mathbb{R}^{\mathbb{N}^n_{2k}}$
for some $k$ (cf. \cite{Nie}).
By (\ref{ECPee}),   a   matrix $A$ is CP  if and only if
 its identifying vector $a$ has a flat extension.

\subsection{Linear optimization with moments}\label{linCone}
Let $\mathcal {A}$ and $K$ be given as (\ref{AE}) and (\ref{K2}), respectively.
Denote
 $$
 \mathscr{R}_\mathcal{A}(K)=\{a\in \mathbb{R}^\mathcal{A}: meas (a,K)\neq \emptyset\},
 $$
where $meas(a,K)$ is the set of all $K$-measures admitted by $a$.
By \eqref{ECPee},   $\mathscr{R}_\mathcal{A}(K)$ is the CP cone (cf. \cite{Nie1}).

Suppose $A\in \mathcal{S}_n$  and $C \in \text{int}(\mathcal{C}_n)$.
Let $a, c\in \mathscr{R}_\mathcal{A}(K)$ be the identifying vectors of $A$
and $C$, respectively.
Replacing $\mathcal{C}_n$ by $\mathscr{R}_\mathcal{A}(K)$, we formulate ($P_1$) as the linear optimization problem with the cone of moments:
\begin{equation*}
  (P_2):\qquad \left\{\begin{array}{lll}
  f_2^*=&\max  & \lambda\\
  &\mbox{s.t.} &   a - \lambda c \in \mathscr{R}_\mathcal{A}(K).
 \end{array}\right.\qquad
\end{equation*}
Similar to Theorem \ref{strongdualP}, we have:
\begin{proposition}\label{strongdualP2}
Suppose $A\in \mathcal{S}_n$ and $C \in \text{int}(\mathcal{C}_n)$.
Then, the optimum $f_2^*$ of ($P_2$) is finite.
Furthermore,
if $f_2^*<0$,
then $A \notin \mathcal{C}_n$; if $f_2^*=0$,
then $A \in \text{bd}(\mathcal{C}_n)$;
 if $f_2^*>0$, then $A \in \text{int}(\mathcal{C}_n)$.
\end{proposition}

Actually, we can further formulate ($P_2$) in the form with  $\mathscr{R}_\mathcal{A}(K)$
and some linear constraints.
Obviously, $c\not=0$.
Suppose $\{p_1, \cdots, p_{\bar n}\}$ is a basis of the orthogonal complement of span$\{c\}$, where $\bar n=\frac{n(n+1)}{2}-1$. Let
\begin{equation}\label{p2b}
p_0=(c^T c)^{-1} c.
\end{equation}
Then,
\begin{equation}\label{orth}
 \langle p_0, c\rangle=1, \quad \langle p_i, c\rangle=0, i=1,\cdots,\bar n.
\end{equation}
Hence, $z=a-\lambda c$ for some $\lambda$ if and only if
\begin{equation}\label{p1}
p_i^T z=p_i^T a, \quad i=1,\cdots,\bar n.
\end{equation}
Moreover,
\begin{equation}\label{p2}
\lambda= (c^T c)^{-1} c^T (a-z).
\end{equation}
The vectors $p_i$ can also be considered as polynomials in $\mathbb{R}[x]_\mathcal{A}$.
Note that
\begin{equation}\label{p3}
\langle p_0, z\rangle= (c^T c)^{-1} c^T z=-\lambda+(c^T c)^{-1} c^T a.
\end{equation}
 By \eqref{p1}-\eqref{p3}, we know ($P_2$) is equivalent to
\begin{equation*}
  (P_3):\qquad\left\{\begin{array}{lll}
 f_3^*= &\min   & \langle p_0,z \rangle\\
  &\mbox{s.t.} & \langle p_i, z\rangle=p_i^T a, \; i=1,\cdots,\bar n\\
    &          & z \in \mathscr{R}_\mathcal{A}(K).
 \end{array}\right.
\end{equation*}

\begin{proposition}\label{PD3}
Suppose $A\in \mathcal{S}_n$ and $C\in \text{int}(\mathcal{C}_n)$.
If $z^*$ is a minimizer of ($P_3$), then
\begin{equation}
\lambda^*= (c^T c)^{-1} c^T (a-z^*)
\end{equation}
is a maximizer of ($P_2$), and vice versa.

\end{proposition}

\section{A semidefinite algorithm for checking interiors}
In this section, we present
a semidefinite algorithm for checking the membership and interiors of $\mathcal{C}_n$.
The cone $\mathscr{R}_\mathcal{A}(K)$ is typically difficult to describe.
However, it has nice semidefinite relaxations.

Let $h$ and $g$ be as in \eqref{K2}.
For each $k \in \mathbb{N}$, denote
\begin{equation}
\label{Ekh}
  \Gamma_{k} (h,g) = \left\{ y \in \mathbb{R}^{\mathbb{N}^n_{2k}}
  :  L^{(k)}_{h} (y) = 0, L^{(k)}_{g_j} (y) \succeq 0, j=0,1,\cdots,n
  \right\}.
\end{equation}
By \eqref{SDPC} and \eqref{RC}, we relax $\mathscr{R}_\mathcal{A}(K)$ by $\Gamma_{k} (h,g)$.
Then the $k$-th order relaxation of ($P_2$) is
\begin{equation*}
 (P_2^k):\qquad \left\{ \begin{array}{lll}
  f_2^k=&\max\limits_{\lambda,y}  & \lambda\\
  &\mbox{s.t.} &   a - \lambda c =y|_{\mathcal{A}},\; y \in  \Gamma_{k} (h,g).
 \end{array}\right.
\end{equation*}
Since $ \text{Feas}(P_2)\subseteq\text{Feas}(P_2^k)$,
we have $f_2^k\geq f_2^*$ for all $k$.
If $f_2^k<0$, then, by Theorem \ref{strongdualP}, $A\notin \mathcal{C}_n$.
Let $\lambda^{*,k}$ be the maximizer of ($P_2^k$).
If $a(\lambda^{*,k}):=a-\lambda^{*,k} c \in \mathscr{R}_\mathcal{A}(K)$, then $f_2^*=f_2^k$ and $\lambda^{*,k}$ is the
maximizer of ($P_2$), i.e., the relaxation ($P_2^k$) is tight for ($P_2$).
If $f_2^k=0$, then $A\in \text{bd}(\mathcal{C}_n)$; otherwise $A\in \text{int}(\mathcal{C}_n)$.

Based on the above, we propose a semidefinite algorithm
for checking interiors of $\mathcal{C}_n$.

\balg\label{Algorithm}  An algorithm for checking interiors of $\mathcal{C}_n$.

\noindent\textbf{Input:} $A\in \mathcal{S}_n$ and $K$ as (\ref{KE}).\\
\textbf{Output:}  An answer $A\notin \mathcal{C}_n$, or $A\in \text{bd}(\mathcal{C}_n)$  or $A\in \text{int}(\mathcal{C}_n)$, with a CP-decomposition.\\
\textbf{Procedure:}

\textbf{Step 0:} Let $k := 1$.

\textbf{Step 1:}
Compute an optimal pair $(\lambda^{*,k},y^{*,k})$ of ($P_2^k$).

\textbf{Step 2:} If $f_2^k<0$, output $A\notin \mathcal{C}_n$ and stop. Otherwise, let $t :=1$.

\textbf{Step 3:} Let $v := y^{*,k}|_{2t}$. If the rank condition (\ref{RC})
is not satisfied, go to Step 6.

\textbf{Step 4:} If $f_2^k=0$, output $A\in \text{bd}(\mathcal{C}_n)$ and stop. Otherwise, go to Step 5.

\textbf{Step 5:} Compute the finitely atomic measure $\mu$ admitted by $v$:
$$\mu = \rho_1 \delta(b_1) + \cdots + \rho_m \delta(b_m),$$
where $m = \text{rank} (M_t (v))$, $b_i \in K$, $\rho_i >
0$, and $\delta(b_i)$ is the Dirac
measure supported on the points $b_i \in K$. Output $A\in \text{int}(\mathcal{C}_n)$
with a CP-decomposition of $A$ \eqref{ECPe} and stop.

\textbf{Step 6:} If $t < k$, set $t := t+1$ and go to Step 3; otherwise, set
$k := k+1$ and go to Step 1.
 \ealg

 Algorithm \ref{Algorithm} gives a certificate for whether $A\notin \mathcal{C}_n$,
 or  $A\in \text{bd}(\mathcal{C}_n)$, or $A\in \text{int}(\mathcal{C}_n)$.
 A CP-decomposition can also be obtained if $A\in \mathcal{C}_n$.

  \begin{remark} \label{matomic}
 We use Henrion and Lasserre's method \cite{HenrionJ} to get a $m$-atomic $K$-measure for $y^{*,k}$.
 The CP-decomposition of the boundary point $A-\lambda^* C$ (see Proposition \ref{bdminimizer}) is computed,
  with which the CP-decomposition of $A$ can be further obtained
 if $A\in \mathcal{C}_n$ (i.e. $\lambda^*\geq 0$).
\end{remark}

\begin{remark}\label{remarkbothal}
We apply Step 3 - Step 6 to check whether
$a(\lambda^{*,k}):= a-\lambda^{*,k} c \in \mathscr{R}_\mathcal{A}(K)$ or not.
It might be possible that $a(\lambda^{*,k})$
belongs to $\mathscr{R}_\mathcal{A}(K)$ while $y^{*,k}|_{2t}$ is not flat for all $t$. In such cases,
we can apply Algorithms given in \cite{Nie,ZhouF} to check if $a(\lambda^{*,k})\in \mathscr{R}_\mathcal{A}(K)$ or not.
In computational experiments, the finite convergence always occurs.

\end{remark}

\begin{remark} In Step 1, we solve   ($P_2^k$).
By Proposition \ref{PD3},
we can instead  solve the relaxation of ($P_3$):
 \begin{equation*}
 (P_3^k):\qquad \left\{ \begin{array}{lll}
 f_3^k= &\min   & \langle p_0,z \rangle\\
  &\mbox{s.t.} & \langle p_i, z\rangle=p_i^T a, \; i=1,\cdots,\bar n\\
    &          & z =y|_{\mathcal{A}},\; y\in \Gamma_{k} (h,g).
 \end{array}\right.
\end{equation*}
\end{remark}

\begin{proposition}\label{PD3k}
Suppose $A\in \mathcal{S}_n$  and $C\in \text{int}(\mathcal{C}_n)$.
If $z^{*,k}$ is a minimizer of ($P_3^k$), then
\begin{equation}
\lambda^{*,k}= (c^T c)^{-1} c^T (a-z^{*,k})
\end{equation}
is a maximizer of ($P_2^k$), and vice versa.
\end{proposition}

Since $K$ as in (\ref{K2}) is nonempty compact and $\mathcal {A}$ as in (\ref{AE}) is finite, $\mathbb{R}[x]_{\mathcal {A}}$ is $K$-full (cf. \cite{ZhouF}).
Note that ($P_2$) always has feasible points.
Combining Nie's result \cite[Theorem 4.3]{Nie1} with Proposition \ref{PD3} and Proposition \ref{PD3k},
we have the following asymptotic convergence of Algorithm \ref{Algorithm}.

\begin{proposition}
\label{Algorithmresults}
Algorithm \ref{Algorithm} has the following properties:
\begin{itemize}

\item [(i)] For all $k$ sufficiently large, ($P_2^k$) has a
maximizing pair $(\lambda^{*,k}, y^{*,k})$.

\item [(ii)] The sequence $\{\lambda^{*,k}\}$ is monotonically decreasing
and converges to the maximizer of ($P_2$). Furthermore, the sequence $\{\lambda^{*,k}\}$ is bounded, and each of its accumulation points is a
maximizer of ($P_2$).
\end{itemize}
\end{proposition}

The finite convergence also happens, under some general conditions in optimization \cite{Nie1}.

\section{Numerical experiments}

In this section, we present numerical experiments
for checking the membership and interiors of $\mathcal{C}_n$ by using Algorithm \ref{Algorithm}.
A CP-decomposition of a matrix is also given if it is CP.
We use softwares GloptiPoly 3 \cite{HenrionJJ} and SeDuMi \cite{Sturm}
to solve   ($P_3^k$) in Step 1 of Algorithm \ref{Algorithm}.
If $|\lambda^{*,k}|< 10^{-4}$, we regard that the matrix is on the boundary of $\mathcal{C}_n$.

\bex\label{BermanExample} \upshape
Consider the matrix $A$ given as (cf. \cite[Example 2.9]{BermanN}):
\begin{equation}\label{BermanExa}
A=\left(\begin{array}{ccccc}
  1&1&0&0&1\\
  1&2&1&0&0\\
  0&1&2&1&0\\
  0&0&1&2&1\\
  1&0&0&1&6
\end{array}\right).
\end{equation}
We have $A\notin \mathcal{C}_5$ (cf. \cite{BermanN}).
We apply Algorithm \ref{Algorithm} to verify this fact.
Choose $C=I_5+E_5$. Then the identifying vector of $C$ is
\[c=(2,1,1,1,1,2,1,1,1,2,1,1,2,1,2)^T.\]
We can choose
\[
p_i=-e_1+e_{i+1}, \quad i\in T_1= \left\{5,9,12,14\right\},\]
\[
p_i=-e_1+2e_{i+1},\quad i\in \{1,\ldots,14\} \setminus T_1
\]
to be basis vectors of span$\{c\}^{\perp}$. Let
\[p_0=(c^T c)^{-1} c=\frac{1}{30}\cdot c.\]
Since $\lambda^{*,k}=-0.3982 <0$ at $k=1$,
we know $A\notin \mathcal{C}_5$.
 \eex

\bex\label{Exampleshaked}\upshape
Consider the matrix $A$ given as (similar to \cite[Exercise 2.22]{BermanN}):
\begin{equation}\label{Exashaked}
A=\left(\begin{array}{ccccccc}
  2&1 &0 &0 &0 &0 &1\\
  1&2 &1 &0 &0 &0 &0\\
  0&1 &2 &1 &0 &0 &0\\
  0&0 &1 &2 &1 &0 &0\\
  0&0 &0 &1 &2 &1 &0\\
  0&0 &0 &0 &1 &2 &1\\
  1&0 &0 &0 &0 &1 &2
\end{array}\right).
\end{equation}
It is shown in \cite{BermanN} that nonnegative symmetric diagonally dominant matrices are completely positive. So,  $A\in \mathcal{C}_7$. Since $A\ngtr 0$, we have $A\in \text{bd}(\mathcal{C}_7)$.
We now verify it  by Algorithm \ref{Algorithm}.
Choose $C=I_7+E_7$. Then the identifying vector of $C$ is
\[c=(2,1,1,1,1,1,1,2,1,1,1,1,1,2,1,1,1,1,2,1,1,1,2,1,1,2,1,2)^T.\]
We choose
\[
p_i=-e_1+e_{i+1}, \quad i\in T_2= \left\{7,13,18,22,25,27\right\},\]
\[
p_i=-e_1+2e_{i+1},\quad i\in \{1,\ldots,27\} \setminus T_2
\]
to be basis vectors of span$\{c\}^\perp$. Let
\[p_0=(c^T c)^{-1} c=\frac{1}{49}\cdot c.\]
Algorithm \ref{Algorithm} terminates at $k=4$, with
$|\lambda^{*,k}|= 2.0815e-008<10^{-4}$ and $y(\lambda^{*,k})\in \mathscr{R}_\mathcal{A}(K)$. As $\lambda^{*,k}\approx 0$, we regard $A\in \text{bd}(\mathcal{C}_7)$.
We obtain the  CP-decomposition $A =\sum_{i=1}^{7} \rho_i b_i b_i^T$,
where the points and their weights are:
{\small
\begin{eqnarray}\label{Exadur3}
&\rho_1=2.0000,  & b_1=(0.0000, 0.0000, 0.0000, 0.0000, 0.0000, 0.7071, 0.7071)^T, \nonumber \\
&\rho_2=2.0000,  & b_2=(0.0000, 0.0000, 0.0000, 0.0000, 0.7071, 0.7071, 0.0000)^T, \nonumber \\
&\rho_3=2.0000,  & b_3=(0.7071, 0.0000, 0.0000, 0.0000, 0.0000, 0.0000, 0.7071)^T, \nonumber\\
&\rho_4=2.0000,  & b_4=(0.0000, 0.0000, 0.7071, 0.7071, 0.0000, 0.0000, 0.0000)^T, \nonumber\\
&\rho_5=2.0000,  & b_5=(0.0000, 0.7071, 0.7071, 0.0000, 0.0000, 0.0000, 0.0000)^T,\nonumber\\
&\rho_6=2.0000,  & b_6=(0.0000, 0.0000, 0.0000, 0.7071, 0.7071, 0.0000, 0.0000)^T, \nonumber\\
&\rho_7=2.0000,  & b_7=(0.7071, 0.7071, 0.0000, 0.0000, 0.0000, 0.0000, 0.0000)^T.\nonumber
\end{eqnarray}
}
In fact, we get the minimal CP-decomposition (cf. \cite[Remark 3.1]{Shaked-MondererBomze}).
\eex

\bex\label{Exampleself1}\upshape
Consider the matrix $A$ given as:
\begin{equation}\label{Exaself1}
A=\left[\begin{array}{cccccc}
   2 &1 &1  &1  &1  &2\\
   1 &2 &3  &1  &1  &1\\
   1 &3 &6 &4  &1  &1\\
   1 &1 &4  &11 &3  &1\\
   1 &1 &1  &3  &9 &3\\
   2 &1 &1  &1  &3  &3
\end{array}\right].
\end{equation}
Since $A$ can be written as
\begin{align}\label{Exaself1decom}
A =&  \left(\begin{array}{c}
  1 \\
  1 \\
  1 \\
  1 \\
  1 \\
  1
\end{array}\right)\left(\begin{array}{c}
  1 \\
  1 \\
  1 \\
  1 \\
  1 \\
  1
\end{array}\right)^T  + \left(\begin{array}{c}
  0 \\
  1 \\
  2 \\
  0 \\
  0 \\
  0
\end{array}\right)\left(\begin{array}{c}
  0 \\
  1 \\
  2 \\
  0 \\
  0 \\
  0
\end{array}\right)^T  + \left(\begin{array}{c}
  0 \\
  0 \\
  1 \\
  3 \\
  0 \\
  0
\end{array}\right)\left(\begin{array}{c}
  0 \\
  0 \\
  1 \\
  3 \\
  0 \\
  0
\end{array}\right)^T\\
& + \left(\begin{array}{c}
  0 \\
  0 \\
  0 \\
  1 \\
  2 \\
  0
\end{array}\right)\left(\begin{array}{c}
  0 \\
  0 \\
  0 \\
  1 \\
  2 \\
  0
\end{array}\right)^T+\left(\begin{array}{c}
  0 \\
  0 \\
  0 \\
  0 \\
  2 \\
  1
\end{array}\right)\left(\begin{array}{c}
  0 \\
  0 \\
  0 \\
  0 \\
  2 \\
  1
\end{array}\right)^T+\left(\begin{array}{c}
  1 \\
  0 \\
  0 \\
  0 \\
  0 \\
  1
\end{array}\right)\left(\begin{array}{c}
  1 \\
  0 \\
  0 \\
  0 \\
  0 \\
  1
\end{array}\right)^T,\nonumber
\end{align}
by  Dickinson's result (\ref{dickinsonform}), $A\in \text{int}(\mathcal{C}_6)$.
We now verify it by Algorithm \ref{Algorithm}.
Choose $C=I_6+E_6$. Then the identifying vector of $C$ is
\[c=(2,1,1,1,1,1,2,1,1,1,1,2,1,1,1,2,1,1,2,1,2)^T.\]
Choose
\[
p_i=-e_1+e_{i+1}, \quad i\in T_1= \left\{6,11,15,18,20\right\},\]
\[
p_i=-e_1+2e_{i+1},\quad i\in \{1,\ldots,20\} \setminus T_1
\]
to be basis vectors of span$\{c\}^\perp$. Let
\[p_0=(c^T c)^{-1} c=\frac{1}{39}\cdot c.\]
Algorithm \ref{Algorithm} terminates at $k=3$, with
$\lambda^{*,k}=0.0726>0$. So, $A\in \text{int}(\mathcal{C}_6)$.
We obtain the CP-decomposition $A =0.0726\cdot(I_6 +E_6)+\sum_{i=1}^{7} \rho_i b_i b_i^T$,
where
\begin{eqnarray}
&\rho_1=2.9447,  & b_1=(0.1034, 0.0000, 0.1929, 0.9757, 0.0000, 0.0000)^T, \nonumber \\
&\rho_2=5.5366,  &  b_2=(0.0945, 0.0561, 0.4340, 0.8734, 0.0000, 0.1918)^T, \nonumber \\
&\rho_3=5.8588, &  b_3=(0.0000, 0.0030, 0.0000, 0.6941, 0.7199, 0.0000)^T,  \nonumber\\
&\rho_4=3.0631,  & b_4=(0.0986, 0.3668, 0.7263, 0.5729, 0.0000, 0.0000)^T, \nonumber \\
&\rho_5=4.1047,  & b_5=(0.0790, 0.5271, 0.8462, 0.0000, 0.0000, 0.0000)^T, \nonumber \\
&\rho_6=2.8900,  &  b_6=(0.7372, 0.2209, 0.0000, 0.0000, 0.0000, 0.6386)^T, \nonumber \\
&\rho_7=7.7308, &  b_7=(0.1383, 0.1364, 0.1383, 0.0000, 0.8676, 0.4365)^T.  \nonumber
\end{eqnarray}

\eex

\section{Dickinson's form}
We present Algorithm \ref{Algorithm} for checking the membership and interiors of  $\mathcal{C}_n$.
If $A\in \mathcal{C}_n$, Algorithm \ref{Algorithm} can give a CP-decomposition of $A$.
Actually, we can also design a similar algorithm to give a
CP-decomposition of $A$ in Dickinson's form if $A\in \text{int}(\mathcal{C}_n)$.

\begin{lemma}\label{iffint} Suppose $A\in \mathcal{S}_n$.
Then $A \in \text{int}(\mathcal{C}_n)$ if and only if rank$(A)=n$  and, for some $b_1\in R_{++}^n$,
there exists a $\lambda>0$ such that $A- \lambda b_1 b_1^T \in \mathcal{C}_n$.
\end{lemma}

The proof of Lemma \ref{iffint} is similar to that of Lemma \ref{iffintDur}, so we omit here.
Lemma \ref{iffint} gives an equivalent characterization of the interior of $\mathcal{C}_n$.
Therefore, we can also transform the problem of checking interiors of $\mathcal{C}_n$ to the following linear optimization problem:
\begin{equation*}
 (\bar P_1):\qquad  \left\{ \begin{array}{lll}
  \bar f_1^*=&\max  & \lambda\\
  &\mbox{s.t.} & A-\lambda b_1 b_1^T \in \mathcal{C}_n,
 \end{array} \right.
\end{equation*}
where $b_1\in R^n_{++}$.
A simple choice of $b_1$ is $\mathbf{1}_{n}$, the $n$ dimensional vector of all ones.
The difference between $(\bar P_1)$ and $(P_1)$ is that we use $b_1b_1^T\in \text{bd}(\mathcal{C}_n)$ instead of $C\in \text{int}(\mathcal{C}_n)$.

By repeating similar arguments as in Sections 2 and 3, we can get

\begin{itemize}
\item[(1)] If $(\bar P_1)$ is infeasible, then $A\notin \mathcal{C}_n$.

\item[(2)] If $(\bar P_1)$ is feasible, we have:

\begin{itemize}

\item[(i)]  If $\bar f_1^*<0$, then $A\notin \mathcal{C}_n$.

\item[(ii)]  If $\bar f_1^*=0$, then $A\in \text{bd}(\mathcal{C}_n)$.

\item[(iii)]  If $\bar f_1^*>0$ and rank$(A)< n$, then $A\in \text{bd}(\mathcal{C}_n)$.

\item[(iv)]  If $\bar f_1^*>0$ and rank$(A)=n$, then $A\in \text{int}(\mathcal{C}_n)$.

\end{itemize}
\end{itemize}

We formulate $(\bar P_1)$ as the linear optimization problem:
\begin{equation*}
  (\bar P_2):\qquad \left\{\begin{array}{lll}
  \bar f_2^*=&\max  & \lambda\\
  &\mbox{s.t.} &   a - \lambda \bar b \in \mathscr{R}_\mathcal{A}(K),
 \end{array}\right.\qquad
\end{equation*}
where $a$ and $\bar b$ are the identifying vectors of $A$ and $b_1b_1^T$, respectively.
The $k$-th order semidefinite relaxation of $(\bar P_2)$ is
\begin{equation*}
 (\bar P_2^k):\qquad \left\{ \begin{array}{lll}
  \bar f_2^k=&\max\limits_{\lambda,y}  & \lambda\\
  &\mbox{s.t.} &   a - \lambda \bar b =y|_{\mathcal{A}},\; y \in  \Gamma_{k} (h,g).
 \end{array}\right.
\end{equation*}

We present another  algorithm for checking the membership and interiors of  $\mathcal{C}_n$ as follows.

\balg\label{Algorithmext} 
\qquad

\noindent\textbf{Input:} $A\in \mathcal{S}_n$ and $K$ as (\ref{KE}).\\
\textbf{Output:}  $A\notin \mathcal{C}_n$, or $A\in \text{bd}(\mathcal{C}_n)$, or $A\in \text{int}(\mathcal{C}_n)$
with a CP-decomposition in Dickinson's form (\ref{dickinsonform}).\\
\textbf{Procedure:}

\textbf{Step 0:} Let $k := 1$.

\textbf{Step 1:} Solve the relaxation ($\bar P_2^k$). If ($\bar P_2^k$) is infeasible, stop and output that $A\notin \mathcal{C}_n$;
otherwise, compute an optimal pair $(\lambda^{*,k},y^{*,k})$ of ($P_2^k$).

\textbf{Step 2:} If $\bar f_2^k<0$, stop and output that $A\notin \mathcal{C}_n$;
else let $t :=1$.

\textbf{Step 3:} Let $v := y^{*,k}|_{2t}$. If the rank condition (\ref{RC})
is not satisfied, go to Step 6.

 \textbf{Step 4:} Compute the finitely atomic measure $\mu$ admitted by $v$:
$$\mu = \rho_2 \delta(b_2) + \cdots + \rho_m \delta(b_m),$$
where $m = \text{rank} (M_t (v))$, $b_i \in K$, $\rho_i >
0$, and $\delta(b_i)$ is the Dirac
measure supported on the point $b_i\in K$.

 \textbf{Step 5:} If rank$(A)< n$ or $f_2^k=0$, output $A\in \text{bd}(\mathcal{C}_n)$ with a CP-decomposition and stop.
Otherwise, output $A\in \text{int}(\mathcal{C}_n)$
with a CP-decomposition of $A$ in Dickinson's form \eqref{dickinsonform} and stop.

\textbf{Step 6:} If $t < k$, set $t := t+1$ and go to Step 3; otherwise, set
$k := k+1$ and go to Step 1.
 \ealg

 Algorithm \ref{Algorithmext} can check whether a matrix $A\in \mathcal{S}_n$ is CP or not.
If it is CP, Algorithm \ref{Algorithmext}  can further  check whether $A\in \text{bd}(\mathcal{C}_n)$ or $A\in \text{int}(\mathcal{C}_n)$.
If  $A\in \text{int}(\mathcal{C}_n)$, a CP-decomposition of $A$ in Dickinson's form (\ref{dickinsonform}) can be given.
The convergence results of Algorithm \ref{Algorithmext} are similar to those of Algorithm \ref{Algorithm}, so we omit here.

We test Algorithm \ref{Algorithmext} on  some examples.

\bex\label{Exampleself1dickson}\upshape
 Consider the   matrix $A$ given as \eqref{Exaself1}.
We now use Algorithm \ref{Algorithmext} to verify $A\in \text{int}(\mathcal{C}_6)$.
Let $b_1=\mathbf{1}_{6}$. Then the identifying vector of $b_1 b_1^T$ is $\bar b=\mathbf{1}_{21}$.
\[\bar b=(1,1,1,1,1,1,1,1,1,1,1,1,1,1,1,1,1,1,1,1,1)^T.\]
Choose
\[
p_i=-e_1+e_{i+1}, \quad i\in \{1,\ldots,20\}
\]
to be basis vectors of span$\{\bar b\}^{\perp}$, and let
\[p_0=(\bar b^T \bar b)^{-1} \bar b=\frac{1}{21}\cdot \bar b.\]
Algorithm \ref{Algorithmext} terminates at $k=3$, with
$\lambda^{*,k}=1.0000>0$ and $y(\lambda^{*,k})\in \mathscr{R}_\mathcal{A}(K)$. So, $A\in \text{int}(\mathcal{C}_6)$.
We  obtain the  CP-decomposition $A =\sum_{i=1}^{6} \rho_i b_i b_i^T$ in Dickinson's form,
where
\begin{eqnarray}
&\rho_1=1.0000,  & b_1=(1.0000, 1.0000, 1.0000, 1.0000, 1.0000, 1.0000)^T, \nonumber \\
&\rho_2=5.0000,  &  b_2=(0.0000, 0.4472, 0.8944, 0.0000, 0.0000, 0.0000)^T, \nonumber \\
&\rho_3=10.0000, &  b_3=(0.0000, 0.0000, 0.3162, 0.9487, 0.0000, 0.0000)^T,  \nonumber\\
&\rho_4=5.0000,  & b_4=(0.0000, 0.0000, 0.0000, 0.0000, 0.8944, 0.4472)^T, \nonumber \\
&\rho_5=2.0000,  & b_5=(0.7071, 0.0000, 0.0000, 0.0000, 0.0000, 0.7071)^T, \nonumber \\
&\rho_6=5.0000,  &  b_6=(0.0000, 0.0000, 0.0000, 0.4472, 0.8944, 0.0000)^T. \nonumber
\end{eqnarray}
The computed decomposition above is the same as \eqref{Exaself1decom}.
\eex

\bex\label{Example6}\upshape
Consider the matrix $A$ given as (cf. \cite{Sponseldur}):
\begin{equation}\label{Exadur}
A=\left(\begin{array}{ccccc}
  2&1 &1 &1 &2\\
  2&2 &2 &1 &1\\
  1&2 &6 &5 &1\\
  1&1 &5 &6 &2\\
  2&1 &1 &2 &3
\end{array}\right).
\end{equation}
Since $A$ can be written as
\begin{equation}\label{Exadur2}
\begin{aligned}
A =&  \left(\begin{array}{c}
  1 \\
  1 \\
  1 \\
  1 \\
  1
\end{array}\right)\left(\begin{array}{c}
  1 \\
  1 \\
  1 \\
  1 \\
  1
\end{array}\right)^T  + \left(\begin{array}{c}
  0 \\
  1 \\
  1 \\
  0 \\
  0
\end{array}\right)\left(\begin{array}{c}
  0 \\
  1 \\
  1 \\
  0 \\
  0
\end{array}\right)^T  + \left(\begin{array}{c}
  0 \\
  0 \\
  2 \\
  2 \\
  0
\end{array}\right)\left(\begin{array}{c}
  0 \\
  0 \\
  2 \\
  2 \\
  0
\end{array}\right)^T\\
& + \left(\begin{array}{c}
  0 \\
  0 \\
  0 \\
  1 \\
  1
\end{array}\right)\left(\begin{array}{c}
  0 \\
  0 \\
  0 \\
  1 \\
  1
\end{array}\right)^T+\left(\begin{array}{c}
  1 \\
  0 \\
  0 \\
  0 \\
  1
\end{array}\right)\left(\begin{array}{c}
  1 \\
  0 \\
  0 \\
  0 \\
  1
\end{array}\right)^T,
\end{aligned}
\end{equation}
by  Dickinson's form (\ref{dickinsonform}), $A\in \text{int}(\mathcal{C}_5)$.
Moreover, the decomposition above is minimal (cf. \cite{Dickinson2010}).
We verify $A\in \text{int}(\mathcal{C}_5)$ by Algorithm \ref{Algorithmext}.
Choose $b_1=\mathbf{1}_{5}$. Then the identifying vector of $b_1b_1^T$ is $\bar b=\mathbf{1}_{15}$.
Choose
$$p_i=-e_1+e_{i+1}, \quad i=1,\ldots,14$$
to be basis vectors of span$\{\bar b\}^{\perp}$, and let
$$
p_0=(\bar b^T \bar b)^{-1} \bar b=\frac{1}{15}\cdot\mathbf{1}_{15}.
$$
Algorithm \ref{Algorithmext} terminates at $k=3$, with
$\lambda^{*,k}=1.0000>0$ and $y(\lambda^{*,k})\in \mathscr{R}_\mathcal{A}(K)$.
So, $A\in \text{int}(\mathcal{C}_5)$. We  obtain the CP-decomposition
$A =\sum_{i=1}^{5} \rho_i b_i b_i^T$ in Dickinson's form,
 where
\begin{eqnarray}\label{Exadur3}
&\rho_1=1.0000,  & b_1=(1.0000, 1.0000, 1.0000, 1.0000, 1.0000)^T, \nonumber \\
&\rho_2=2.0000,  &  b_2=(0.0000, 0.0000, 0.0000, 0.7071, 0.7071)^T, \nonumber \\
&\rho_3=8.0000, &  b_3=(0.0000, 0.0000, 0.7071, 0.7071, 0.0000)^T,  \\
&\rho_4=2.0000,  & b_4=(0.7071, 0.0000, 0.0000, 0.0000, 0.7071)^T, \nonumber\\
&\rho_5=2.0000,  & b_5=(0.0000, 0.7071, 0.7071, 0.0000, 0.0000)^T.\nonumber
\end{eqnarray}
The computed decomposition \eqref{Exadur3} is the same as \eqref{Exadur2}. We get a minimum CP-decomposition.

\eex


\end{document}